\documentclass[a4paper,twoside,11pt,reqno]{amsart}
\newcommand{\Ueberschrift}{Convex Fujita numbers: projective bundles}
\newcommand{\Kurztitel}{Convex Fujita numbers: projective bundles}
\usepackage{amsmath} \usepackage{amssymb} \usepackage{amsopn}
\usepackage{amsthm} \usepackage{amsfonts} \usepackage{mathrsfs}
\usepackage{latexsym}
\usepackage[headings]{fullpage}
\usepackage{mathtools}
\usepackage[all]{xy}
\usepackage[dvips]{graphicx} 
\usepackage[usenames]{color}
\usepackage[T1]{fontenc}
\usepackage[utf8]{inputenc}
\usepackage{mathabx}
\usepackage{stmaryrd}

\usepackage{setspace,fancyhdr}
\usepackage[margin=2.5cm]{geometry}
\usepackage[alphabetic]{amsrefs}
\usepackage{framed}
\usepackage{tikz}
\usepackage{enumitem}
\usepackage{booktabs}
\usepackage[pdfpagelabels, pdftex]{hyperref}
\hypersetup{
  pdftitle={},
  pdfauthor={},
  pdfsubject={},
  pdfkeywords={},
  colorlinks=true,    
  linkcolor=blue,     
  citecolor=blue,     
  filecolor=blue,      
  urlcolor=blue,       
  breaklinks=true,
  bookmarksopen=true,
  bookmarksnumbered=true,
  pdfpagemode=UseOutlines,
  plainpages=false}
\DeclareMathOperator{\rH}{H}

\DeclareMathOperator{\rM}{M}
\DeclareMathOperator{\rN}{N}

\DeclareMathOperator{\rc}{c}

\newcommand{\bC}{{\mathbb C}}

\newcommand{\bP}{{\mathbb P}}
\newcommand{\bQ}{{\mathbb Q}}
\newcommand{\bR}{{\mathbb R}}

\newcommand{\bZ}{{\mathbb Z}}

\newcommand{\cE}{{\mathscr E}}
\newcommand{\cF}{{\mathscr F}}
\newcommand{\cG}{{\mathscr G}}
\newcommand{\cH}{{\mathscr H}}

\newcommand{\cL}{{\mathscr L}}
\newcommand{\cM}{{\mathscr M}}

\newcommand{\cO}{{\mathscr O}}

\newcommand{\cU}{{\mathscr U}}

\newcommand{\dO}{{\mathcal O}}

\newcommand{\surj}{\twoheadrightarrow}

\DeclareMathOperator{\pr}{pr}

\DeclareMathOperator{\End}{End}

\DeclareMathOperator{\GL}{GL}
\DeclareMathOperator{\PGL}{PGL}
\DeclareMathOperator{\SL}{SL}

\DeclareMathOperator{\Sp}{Sp}

\DeclareMathOperator{\Nef}{Nef}

\DeclareMathOperator{\Pic}{Pic}

\newcommand{\norm}{{\rm norm}}

\DeclareMathOperator{\Alb}{Alb}

\DeclareMathOperator{\NS}{NS}

\DeclareMathOperator{\Ext}{Ext}

\DeclareMathOperator{\cHom}{\cH\mathit{om}}

\DeclareMathOperator{\res}{res}

\DeclareMathOperator{\ch}{ch}

\newcommand{\tors}{{\rm tors}}

\DeclareMathOperator{\Sym}{Sym}

\DeclareMathOperator{\rank}{rk}

\newcommand{\Albanesefinite}{of abelian Picard type} 

\newtheorem{thm}{Theorem}[section]

\newtheorem{prop}[thm]{Proposition}
\newtheorem{lem}[thm]{Lemma}

\newtheorem{cor}[thm]{Corollary}

\newtheorem{thmABC}{Theorem}

\theoremstyle{definition}
\newtheorem{defi}[thm]{Definition}

\theoremstyle{remark}
\newtheorem{rmk}[thm]{Remark}
\newtheorem{rmks}[thm]{Remarks}

\newtheorem{ex}[thm]{Example}

\newenvironment{pro*}[1][Proof]{{\it{#1:}} }{}
\newenvironment{pro**}[1][]{{\it{#1}} }{\hfill $\square$}

\numberwithin{equation}{section}
\newcommand{\tref}[1]{Theorem~\ref{#1}}

\newcommand{\cref}[1]{Corollary~\ref{#1}}

\newcommand{\dref}[1]{Definition~\ref{#1}}

\newcommand{\lref}[1]{Lemma~\ref{#1}}
\newcommand{\pref}[1]{Proposition~\ref{#1}}

\def\mc{\mathscr}

\def\rank{\text{rank}}

\def\pr{\text{pr}}

\def\P{\mathbb{P}}
\def\bP{\P}

\def\cE{\mc{E}}
\def\cF{\mc{F}}
\def\cG{\mc{G}}
\def\cH{\mc{H}}

\def\cL{\mc{L}}
\def\cM{\mc{M}}

\def\cO{\mc{O}}

\def\cU{\mc{U}}

\DeclareMathOperator{\conFN}{Fu} 

\definecolor{intOrange}{rgb}{1.0,.310,.0}
\usepackage[textsize=tiny]{todonotes}

\begin{document}

\hrule width\hsize

\vskip 0.5cm

\title[\Kurztitel]{\Ueberschrift} 

\author{Jiaming Chen}
\address{Jiaming Chen, Institut f\"ur Mathematik, Goethe--Universit\"at Frankfurt, Robert-Mayer-Stra\ss e {6--8},
60325~Frankfurt am Main, Germany} 
\email{\tt chen@math.uni-frankfurt.de}

\author{Alex K\"{u}ronya}
\address{Alex K\"uronya, Institut f\"ur Mathematik, Goethe--Universit\"at Frankfurt, Robert-Mayer-Stra\ss e {6--8},
60325~Frankfurt am Main, Germany} 
\email{\tt kuronya@math.uni-frankfurt.de}

\author{Yusuf Mustopa}
\address{Yusuf Mustopa, University of Massachusetts Boston, Department of Mathematics, Wheatley Hall, 100 William T Morrissey Blvd, Boston, MA 02125, USA}
\email{Yusuf.Mustopa@umb.edu}

\author{Jakob Stix}
\address{Jakob Stix, Institut f\"ur Mathematik, Goethe--Universit\"at Frankfurt, Robert-Mayer-Stra\ss e {6--8},
60325~Frankfurt am Main, Germany} 
\email{\tt stix@math.uni-frankfurt.de} 
	
\thanks{The authors acknowledge support by Deutsche Forschungsgemeinschaft  (DFG) through the Collaborative Research Centre TRR 326 "Geometry and Arithmetic of Uniformized Structures", project number 444845124.}
	
\maketitle

\date{\today} 

\maketitle

\begin{quotation} 
\noindent \small {\bf Abstract} --- We study effective global generation properties of projectivizations of  curve semistable vector bundles over curves and abelian varieties. 
\end{quotation}

\DeclareRobustCommand{\SkipTocEntry}[5]{}
\setcounter{tocdepth}{1} {\scriptsize \tableofcontents}

\section{Introduction}
\label{sec:intro}
\subsection{Motivation} 

This work is a natural continuation of our earlier efforts to study effective global generation of line bundles beyond the precision provided by Fujita's conjecture. While the articles \cite{ChenKuronyaMustopaStix2023:ConvexFujitaFundgp} and \cite{ChenKuronyaMustopaStix2023:ConvexFujitaSurfaces} studied possible connections with the fundamental group of the underlying variety, and the correlation with the Kodaira--Enriques classification of surfaces, here we contemplate the situation for projective bundles associated with curve semistable vector bundles. 

The circle of ideas around Fujita's freeness conjecture originated in the article \cite{Fujita1987:PolarizedManifoldsWhose}, where the author proved that, given a smooth complex projective variety $X$, any line bundle of the form $\omega_X\otimes \cL^{m}$ is nef where $\cL$ is an ample line bundle, and $m\geq \dim\,X+1$. At the same time, he conjectured what later came to be known as Fujita's freeness conjecture, namely, that the line bundles $\omega_X\otimes \cL^{m}$ (for $\cL$ ample and $m\geq \dim\,X+1$) are globally generated as well. 

While the conjecture follows quickly from Riemann--Roch on curves, it is non-trivial in higher dimensions, and has given rise to a lot of novel mathematics. It  has been demonstrated for $\dim(X)\leq 6$ \cites{Reider1988:VectorBundlesRank,EinLazarsfeld1993:GlobalGenerationPluricanonical,Kawamata1997:FujitaFreenessConjecture,Helmke1997:FujitaConjecture,YeZhu2020:FujitaFreenessConjecture,GhidelliLacini2021:LogarithmicBoundsFujita}. There exist analogous suboptimal results valid in all dimensions,  due to Angehrn--Siu \cite{AngehrnSiu1995:EffectiveFreenessPoint}, Heier \cite{Heier2002:EffectiveFreenessAdjoint},  and Ghidelli and Lacini \cite{GhidelliLacini2021:LogarithmicBoundsFujita}. It is important to point out  that although these bounds are uniform, they are  not linear in $\dim(X)$. 

There undoubtedly exist polarized projective varieties with extremal behaviour from the point of view of Fujita's conjecture (take projective spaces for one), at the same time,  one often meets cases where a substantially smaller multiple will suffice. Hence, there is in general room for improvement, and a need for a more precise analysis. 

Convex Fujita numbers have been introduced to this end in \cite{ChenKuronyaMustopaStix2023:ConvexFujitaFundgp} as a  measure of effective positivity of line bundles on smooth complex projective varieties.  Recall that the convex Fujita number of a variety $X$, which we denote by $\conFN(X),$ is the minimal $m \geq 0$ such that  for all $t \geq m$ and any ample line bundles $\cL_1, \ldots, \cL_t$ on $X$ the adjoint bundle 
\[
\omega_X \otimes \cL_1 \otimes \ldots \otimes \cL_t
\]
is globally generated.  It follows from \cite{AngehrnSiu1995:EffectiveFreenessPoint}*{Thm~0.1} 
that $\conFN(X)$ is always finite, see \cite{ChenKuronyaMustopaStix2023:ConvexFujitaFundgp}*{Prop.~2.5}. 

Our first remark is that most of the major general results mentioned above remain true in their convex formulation (see \cite{ChenKuronyaMustopaStix2023:ConvexFujitaFundgp}*{Section 2}) including the classical theorems of Reider, Helmke, Kawamata, and Angehrn--Siu. The introduction of convex Fujita numbers also led to relations between various configurations of polarized varieties. 
For instance, for varieties  $X$ and $Y$, one has 
\[
\conFN(X \times Y) \geq \max\{\conFN(X), \conFN(Y) \} 
\] 
with equality if  $\Pic^0_X$ and $\Pic^0_Y$ have no common nontrivial isogeny factor.

For each $n \geq 1$ the simplest variety of dimension $n$ with the largest convex Fujita number predicted by the Fujita freeness conjecture is projective space.  On the other hand, we know that $\conFN(X) = 2$ when $X$ is a curve by Riemann-Roch and that ${\rm Fu}(X) \leq 2$ when $X$ is an abelian variety essentially by theorems of Lefschetz.  

In this article we will  focus on a class of varieties that lies in between these examples, namely projectivized vector bundles mainly over curves and abelian varieties. 
One of the key points is to identify a class of vector bundles that resemble line bundles closely. The class of vector bundles  we work with, which we call \emph{curve semistable}, has been studied from the point of view of its differential geometry or semistability, but, as of now, their relation to effective positivity has not been investigated.

\subsection{Results} 
As far as effective global generation of projectivized vector bundles go, one can quickly  check that if $X =\bP(\cE)$  for  a vector bundle $\cE$ of rank $r$ on a smooth projective variety $S$, then $\conFN(X) \geq r$, see \pref{prop:FNonPElowerbound}.  On the other hand, Fujita's freeness conjecture hints at $\conFN(X) \leq r+\dim{S}$.  Our first result addresses the case $\dim{S} = 1$.

\begin{thmABC}[see \tref{thm:FNonPEuppperbound} and \tref{thm:FNofPE Yusufs result}]
\label{thmABC:PEoverCurves}
	Let $C$ be a smooth projective curve of genus $g \geq 2,$ let $\cE$ be a vector bundle of rank $r \geq 2$ and degree $d$ on $C,$ and let $\bP(\cE)$ be its projectivization.  Then we have the following:
	\begin{enumerate} [label=(\arabic*), ref = (\arabic*)]
		\item
		$r \leq \conFN(\bP(\cE)) \leq r+1$.
		\item
		If $\cE$ is not stable, or if it is stable with $(r,d) \neq 1$, then $\conFN(\bP(\cE)) = r$.
		\item
		\label{thmitem:A3}
		If $\cE$ is stable and $(r,d)=1$,  then:
				\begin{enumerate}[label=(\roman*)]
					\item
					\label{thmitem:A3i}
					If $d \nequiv 1~(\rm{mod }~r)$ and $\cE$ is very general in the moduli space $\cU_{C}(r,d)$ parametrizing stable vector bundles of rank $r$ and degree $d$ on $C$, then $\conFN(\bP(\cE)) = r$. 
					\item 
					If $d \equiv 1~(\rm{mod }~r),$ then $\conFN(\bP(\cE)) = r+1$.
				\end{enumerate}
	\end{enumerate}
\end{thmABC}

Most of \tref{thmABC:PEoverCurves} appears 
in \cite[Theorem~3.5]{ChenKuronyaMustopaStix2023:ConvexFujitaSurfaces} and is restated for completeness and comparison with the case of an abelian variety as the base. The new part here is part  \ref{thmitem:A3i} of \ref{thmitem:A3}, the genericity condition of which is quite explicit  in our proof of \tref{thm:FNofPE Yusufs result}. Note that if $r=1$, \tref{thmABC:PEoverCurves} simply states that curves $C$ have convex Fujita number $\conFN(C) = 2$, and that is an easy consequence of Riemann-Roch.

Our next result concerns semihomogeneous vector bundles over abelian varieties.  Note that the hypothesis of the following result is satisfied for a very general polarized abelian variety of fixed dimension and polarization type; moreover, the given upper bound on the convex Fujita number is at least as low as that predicted by the Fujita freeness conjecture. The slope class $\mu(\cE)$ is defined in \dref{defi:slopeclass etc}.

\begin{thmABC}[see \tref{thm:proj-semihom}]
\label{thmABC:sh on AVs}
	Let $A$ be an abelian variety of dimension $g \geq 1$, and let $\cE$ be a 
	semihomogeneous 
	vector bundle on $A$ of rank $r \geq 2$.  Then the convex Fujita number of $X = \bP(\cE)$ satisfies
	\[
	r \leq \conFN(\bP(\cE)) \leq r+1.
	\]
	More precisely, the following holds.
	\begin{enumerate}[align=left,labelindent=0pt,leftmargin=*]
	\item 
	If there is an integer $0 < r' < r$ such that $r' \mu(\cE)$ is the class of a divisor, then $\conFN(\bP(\cE)) = r$.
	\item
	If $\conFN(A) = 0$, i.e. all ample line bundles on $A$ are globally generated, then $\conFN(\bP(\cE)) = r$.
	\item 
	If there is a line bundle $\cM$ on $A$ such that $\det(\cE) \otimes \cM^{\otimes r}$ is ample but not globally generated, then $\conFN(\bP(\cE)) = r+1$.
\end{enumerate}	
\end{thmABC}

Note that if $r=1$ then the conclusion of \tref{thmABC:sh on AVs} does not hold. In fact, abelian varieties $A$ always have $\conFN(A) = 0$ or $2$ essentially by a classical theorem of Lefschetz and proven by Bauer and Szemberg in \cite{BauerSzemberg1996:TensorProductsAmple}*{Theorem 1.1}. In particular the convex Fujita number is never $1$.

\subsection{Curve semistable vector bundles}  Let $S$ be a smooth projective variety. Curve semi\-stable vector bundles on $S$ generalize both semistable vector bundles on curves and semihomogeneous vector bundles on abelian varieties. The class of curve semistable vector bundles is defined to be the largest class of vector bundles that is stable under arbitrary pull back and on curves coincides with all semistable vector bundles. 

\begin{defi}
\label{defi:csst}
	A \textbf{curve semistable} vector bundle is a vector bundle $\cE$ on a proper variety $S$ such that for all maps $f: C \to S$ from a smooth projective curve $C,$ the pull back $f^\ast \cE$ is a semistable vector bundle on $C$. 
\end{defi}

Beside being a suitable first order generalization of line bundles to higher ranks, curve semi\-stable vector bundles have a distinctive flavour because they are an iterated extension of hermitian projectively flat bundles of the same rational slope class, see \tref{thm:projectively flat}. 

We prove in \pref{prop:AV sh csst} that on abelian varieties the notions \emph{curve semistable} and \emph{semihomogeneous} coincide. Moreover, generalizing a result of Miyaoka, we are able to describe in \pref{prop:NEFconePE} the ample cone for $\bP(\cE)$ if $\cE$ is a curve semistable vector bundle, see \eqref{eq:nef Pcsst}.

With curve semistable vector bundles we may work over base varieties which generalize abelian varieties, but still have convex Fujita number less or equal to $2$. The property of having \emph{abelian Picard type} is defined in 
\dref{defi:Albanese-finite}. For the definition of the denominator of a rational numerical class we refer to \dref{defi:denominator}.

\begin{thmABC}[see \tref{thm:FNonPE-Albanesefinite}]
	\label{thmABC:FNonPE-Albanesefinite}
	Let $S$ be a variety \Albanesefinite. We further assume that  the  Albanese map is surjective or the Picard rank of $S$ equals $1$. Let $\cE$ be a curve semistable bundle on $S$ of rank $r \geq 2$ with associated projective bundle $\pi \colon X = \bP(\cE) \to S$. Then
	\[
	r \leq \conFN(X) \leq r+1
	\]
	and the following more precise statements hold:
\begin{enumerate}[align=left,labelindent=0pt,leftmargin=*]
	\item 
	If the denominator of $\mu(\cE)$ is less than $r$, then $\conFN(\bP(\cE)) = r$.
	\item 
	If $\conFN(S) \leq 1$ and $\pi_1(S)$ is abelian, then $\conFN(\bP(\cE)) = r$.
	\item
	If there is a line bundle $\cM$ on $S$ such that $\det(\cE) \otimes \cM^{\otimes r}$ is ample but not globally generated, then $\conFN(\bP(\cE)) = r+1$.
\end{enumerate}
\end{thmABC}

Note that if $r=1$ then the proof of \tref{thmABC:FNonPE-Albanesefinite} yields $\conFN(S) \leq 2$ for $S$ \Albanesefinite\  with surjective Albanese map or Picard rank equal to $1$. 

\subsection*{Acknowledgements}

We are grateful to Thomas Mettler and Bal\'azs Szendr\H{o}i  for helpful discussions.

\section{Curves semistable vector bundles}
\label{sec:curve semistable}

\subsection{Equivalent descriptions of curve semistable bundles}

Our upper bound of the convex Fujita number requires a class of vector bundles $\cE$ where we are able to determine the ample cone of the projective bundle $\bP(\cE)$. Curve semistable vector bundles have been defined in \dref{defi:csst}.

\begin{rmks}
\begin{enumerate}
	\item
	Since pullback by a finite map of curves preserves semistability, see 
	\cite[Prop.~3.2]{Miyaoka1987:ChernClassesKodaira} or \cite[Lemma~6.4.12]{Lazarsfeld2004:PositivityAlgebraicGeometry}, a vector bundle on a curve is semistable if and only if it is curve  semistable.
	\item
	Curve semistability is preserved under any operation of the following kind: tensor products, duals, internal Hom's, symmetric powers, alternating powers, determinants. All these operations commute with base change and preserve semistability of vector bundles on curves, see \cite[\S3]{Miyaoka1987:ChernClassesKodaira}.
	
	More generally, considering a vector bundle $\cE$ of rank $r$ as a $\GL_r$ torsor, then we can push $\cE$ along an algebraic representation $\rho \colon \GL_r \to \GL_n$ to obtain a vector bundle $\cE^\rho$. If $\rho$ sends scalars to scalars and $\cE$ is curve semistable, then also $\cE^\rho$ is curve semistable; compare \cite[Cor.~3.9]{Miyaoka1987:ChernClassesKodaira} for the assertion on curves from which the assertion about curve semistable bundles follows at once.
	\item
	Direct summands of curve semistable vector bundles are curve semistable, because this is true for semistable vector bundles on curves.
	\item
	The pull back of a curve semistable vector bundle is curve semistable: let $f: Y \to X$ be a map of proper varieties. If  $\cE$ is a curve semistable vector bundle on $X$, then $f^\ast \cE$ is curve semistable on $Y$.
	\item
	Conversely, if $f: Y \to X$ is a surjective map of proper varieties and $\cE$ is a vector bundle on $X$, then $\cE$ is curve semistable if and only if $f^\ast \cE$ is curve semistable on $Y$.  Indeed, if $g: C \to X$ is a map from a proper smooth curve, then there is a branched cover $h: C' \to C$ such that $g \circ h$ lifts to a map $g': C' \to Y$. Then $g^\ast \cE$ is semistable if and only if $h^\ast(g^\ast \cE) = g'^\ast(f^\ast \cE)$ is semistable, a consequence of $f^\ast \cE$ being  curve semistable.
\end{enumerate}
\end{rmks}

Before stating the important general properties of curve semistable bundles, we recall some numerical invariants of vector bundles.

\begin{defi}
\label{defi:slopeclass etc}
	Let $\cE$ be a vector bundle of rank $r$ on a  variety $S$. 
	\begin{enumerate}
		\item
		The \textbf{slope class} or \textbf{average first Chern class} of $\cE$ is the $\bQ$-rational divisor class
		\[
		\mu(\cE) \coloneqq \frac{1}{r}\rc_1(\cE) \in \rN^1(S)_\bQ \ .
		\]
		\item
		The \textbf{normalized tautological class} of $\cE$ is the $\bQ$-divisor class 
		\[
		\lambda_{\cE} \coloneqq \rc_1(\dO(1)) - \pi^\ast \mu(\cE) \,\in \rN^1(\P(\cE))_\bQ \ .
		\]
		\item
		The \textbf{discriminant} of $\cE$ is the class 
		\[
		\Delta(\cE) \coloneqq  2r \cdot \rc_2(\cE) - (r-1) \rc_1(\cE)^2  \in \rH^4(S, \bQ)\ .
		\]
		The discriminant can also be computed as 
		\[
		\Delta(\cE) =  \ch_1(\cE)^2 - 2 \ch_0(\cE) \ch_2(\cE)   = \rc_2(\End(\cE)) \ . 
		\]		
	\end{enumerate}
\end{defi}

The following characterization of curve semistable vector bundles is essentially obtained by Nakayama in 
\cite[Theorem A]{Nakayama1999:NormalizedTautologicalDivisors} with Jahnke and Radloff  \cite[Theorem 1.1]{JahnkeRadloff2013:SemistabilityRestrictedTangent} emphasizing the equivalence to the property of being curve semistable.  

\begin{thm} 
	\label{thm:projectively flat}
	For a vector bundle $\cE$ on a smooth projective variety $S$ of dimension $d$, the following are equivalent:
	\begin{enumerate}[label=(\alph*)]
		\item 
		\label{rmkitem:curvesemistable}
		$\cE$ is curve semistable.
		\item 
		\label{rmkitem:lambdanef}
		The class $\lambda_\cE$ is nef.  
		\item 
		\label{rmkitem:anypolarization}
		$\cE$ is $A$-slope semistable with respect to all ample $A$, 
		and $\Delta(\cE)=0.$
		\item
		\label{rmkitem:discriminant0}
		There is an ample class $A$ such that 
		$\cE$ is $A$-slope semistable, and $\Delta(\cE) = 0$.
		\item
		\label{rmkitem:BogomolovInequalitySharp}
		There is an ample class $A$ such that 
		$\cE$ is $A$-slope semistable, and $\Delta(\cE)\cdot A^{d-2} = 0$.
		\item
		\label{rmkitem:projectivelyflatfiltration}
		There is a filtration by vector bundles 
		\[
		0 = \cE_0 \subset \cE_1 \subset \ldots \cdots \cE_{k-1} \subset \cE_k = \cE
		\]
		such that the $\cF_i := \cE_i/\cE_{i-1}$ can be endowed with a \textbf{hermitian projectively flat} 
		connection, and $\mu(\cF_i) = \mu(\cE)$ as $\bQ$-divisor classes. 
	\end{enumerate}
\end{thm}

\begin{proof}
\cite[Theorem A]{Nakayama1999:NormalizedTautologicalDivisors} shows the equivalence of \ref{rmkitem:lambdanef},  \ref{rmkitem:BogomolovInequalitySharp} and  \ref{rmkitem:projectivelyflatfiltration}. This proof implicitly also shows the equivalence with assertions \ref{rmkitem:curvesemistable} and  \ref{rmkitem:anypolarization}. Since the latter has  been clarified explicitly by other authors, we're going to explain these equivalences together with interesting byproducts of the arguments now. 

\cite[Theorem 1.1]{JahnkeRadloff2013:SemistabilityRestrictedTangent} shows 
\ref{rmkitem:curvesemistable} $\Leftrightarrow$ \ref{rmkitem:lambdanef} as follows. Both \ref{rmkitem:curvesemistable} and  \ref{rmkitem:lambdanef} can be checked by base change to  smooth projective curves. In the curve case, the equivalence was proven 
by Miyaoka in  \cite[Theorem 3.1]{Miyaoka1987:ChernClassesKodaira}  (see also  \cite[Prop.~6.4.11]{Lazarsfeld2004:PositivityAlgebraicGeometry}).

The equivalence of   \ref{rmkitem:lambdanef} and  \ref{rmkitem:anypolarization} is \cite[Theorem 1.4]{BruzzoHernandezRuiperez2006:SemistabilityVsNefness}, adapting the proof from the case of Higgs bundles as explained in \cite[Theorem 1.3]{BruzzoHernandezRuiperez2006:SemistabilityVsNefness} by treating the case of zero Higgs field, and passing through property  \ref{rmkitem:curvesemistable} in the proof\footnote{For a generalization of  \ref{rmkitem:lambdanef} $\Leftrightarrow$  \ref{rmkitem:anypolarization}  to $G/P$-bundles for a reductive group $G$ and a parabolic $P$ see  \cite[Thm.~3.2]{BiswasBruzzo2008:SemistablePrincipalBundles}}.  Here is an alternative replacing the detour via Higgs bundles by the more direct argument based on the Kobayashi-Hitchin correspondence (that ultimately also in the expanded form of the Corlette--Simpson correspondence is the essence of the proof using Higgs bundles). 

Since the implications \ref{rmkitem:anypolarization} $\Rightarrow$ \ref{rmkitem:discriminant0} $\Rightarrow$\ref{rmkitem:BogomolovInequalitySharp} are obvious, it remains to show  \ref{rmkitem:projectivelyflatfiltration} $\Rightarrow$ \ref{rmkitem:anypolarization}, for example. Assuming \ref{rmkitem:projectivelyflatfiltration}, a Chern class computation  that can be found in \cite[II, \S2.3]{Kobayashi1987:DifferentialGeometryComplex} yields 
\begin{equation}
\label{eq:chernch of csst}
\ch(\cE) = \sum_i \ch(\cF_i) = \sum_i \rank(\cF_i) \cdot \exp(\mu(\cF_i)) = \rank(\cE)  \cdot \exp(\mu(\cE)).
\end{equation}
It follows that 
\[
\Delta(\cE) = \ch_1(\cE)^2 - 2 \ch_0(\cE) \ch_2(\cE) = \rc_1(E)^2 - 2 \rank(\cE)^2 \cdot \frac{1}{2} \mu(\cE)^2 = 0.
\]
Let us fix an ample class $A \in \Pic(S)$. 
Since every projectively flat hermitian bundle is weakly Einstein by \cite[Proposition 4.1.11]{Kobayashi1987:DifferentialGeometryComplex}, we can by \cite[Proposition 4.2.4]{Kobayashi1987:DifferentialGeometryComplex} conformally rescale the hermitian metric on $\cF_i$ to yield a Hermite--Einstein metric with respect to the K\"ahler form associated to $\rc_1(A)$. The Kobayashi-Hitchin correspondence as in \cite[Theorem~2.4]{Kobayashi1982:CurvatureStabilityVector} (only one direction is needed) then shows that $\cF_i$ is polystable with $\mu(\cF_i)$ independent of $i$. In particular, it follows that $\cE$ is $A$-slope semistable. This shows \ref{rmkitem:anypolarization} and concludes the proof.
\end{proof}

We write down two notable intermediate statements of the proof as corollaries.
 
\begin{cor}
The Chern character of a curve semistable vector bundle $\cE$ has the form
\[
\ch(\cE) = \rank(\cE)  \cdot \exp(\mu(\cE)).
\]
\end{cor}
\begin{proof}
This is \eqref{eq:chernch of csst}.
\end{proof}

\begin{cor}
\label{cor:JordanHoelder by stable hermitian projectively flat}
A curve semistable vector bundle $\cE$ on $S$  admits a filtration by subbundles that is simultaneously for all ample $A \in \Pic(S)$ a Jordan-H\"older filtration by $A$-stable hermitian projectively flat bundles.
\end{cor}
\begin{proof}
This comes from the following refinement of the filtration asserted in \tref{thm:projectively flat} \ref{rmkitem:projectivelyflatfiltration}. The $\cF_i$ admit a finite decomposition in $\oplus$-indecomposable vector bundles. Since $\cF_i$ is $A$-polystable with respect to any ample $A \in \Pic(S)$, these indecomposable summands must be $A$-stable. We refine the filtration so that the filtration quotients are indecomposable.
\end{proof}

\begin{rmks}
\begin{enumerate}
	\item
	\tref{thm:projectively flat}~\ref{rmkitem:lambdanef}  translates into the
	$\bQ$-rational twisted vector bundle 
	\[
	\cE_{\norm} \coloneq \cE\langle -\mu(\cE)\rangle
	\]
	being nef (by definition).
	\item
	We should not hide the fact that the deepest part of \tref{thm:projectively flat}, namely \ref{rmkitem:BogomolovInequalitySharp} implies \ref{rmkitem:projectivelyflatfiltration} strongly builds on the Kobayashi-Hitchin correspondence and thus work of Donaldson, Mehta and Ramanathan, and Uhlenbeck and Yau. 
\end{enumerate}
\end{rmks} 
 
\begin{defi}
	Recall that a vector bundle $\cE$ on an abelian variety $A$ is \textbf{semihomogeneous} if for all $x \in A$ there exists $\alpha \in {\rm Pic}^{0}(A)$ such that the pull back by translation with $x$ becomes 
	\[
	t_{x}^{\ast}\cE \cong \cE \otimes \alpha.
	\]
\end{defi}

\begin{prop}
\label{prop:AV sh csst}
	Let $\cE$ be a vector bundle on an abelian variety $A$. 
	Then the following are equivalent:
	\begin{enumerate}[label=(\alph*)] 
		\item
		\label{propitem:sh}
		{$\mathcal{E}$ is semihomogeneous.}
		\item
		\label{propitem:cs}
		{$\mathcal{E}$ is curve semistable.}
	\end{enumerate}
\end{prop}

\begin{proof}
	\ref{propitem:sh}  $\Rightarrow$ \ref{propitem:cs}: 
	By \cite[Prop.~6.13]{Mukai1978:SemihomogeneousVectorBundles} a semihomogeneous $\cE$ is Gieseker semistable and thus also slope semistable with respect to any polarization. Let $r$ be the rank of $\cE$. Since by \cite[Lemma 2.6]{KuronyaMustopa2021:EffectiveGlobalGeneration} the Chern character of a semihomogeneous $\cE$ computes as 
	\[
	r \cdot \exp\biggl(\frac{\rc_1(\cE)}{r}\biggr) =   \ch(\cE) =  r + \rc_1(\cE) + \frac{1}{2}\big(\rc_1(\cE)^2 - 2 \rc_2(\cE)\big) + \ldots,
	\]
	we find, by comparing the quadratic terms, 
	\[
	\rc_2(\cE) = \frac{r-1}{2r}\rc_1(\cE)^2,
	\]	
	and thus vanishing discriminant $\Delta(\cE) = 0$. 
	It follows that the semihomogeneous vector bundle $\cE$ is curve semistable due to \tref{thm:projectively flat}  \ref{rmkitem:curvesemistable} $\Leftrightarrow$ \ref{rmkitem:anypolarization}.
		
	\smallskip
	
	\ref{propitem:cs}  $\Rightarrow$ \ref{propitem:sh}:  Assume $\mathcal{E}$ is curve semistable.  Then, as recalled above in \tref{thm:projectively flat}  \ref{rmkitem:projectivelyflatfiltration} and refined in \cref{cor:JordanHoelder by stable hermitian projectively flat}, there exists a filtration
	\[
	0 = \cE_0 \subset \cE_1 \subset \ldots \subset \cE_k = \cE
	\]
	where the successive quotients $\cF_i := \cE_i/\cE_{i-1}$ each are stable (with respect to any polarization on $A$) and admit a projectively flat hermitian connection and satisfy the property that $\mu(\cF_i) = \mu(\cE)$ as classes in $\NS(A) \otimes \bQ$.
	
	A choice of a projectively flat hermitian connection gives the projective bundle $\bP(\cF_i)$ a flat structure, so it is induced by a representation $\pi_1(A) \to \PGL_{n_i}(\bC)$ for $n_i$ equal to the rank of $\cF_i$. Since translation on $A$ induces the identity on $\pi_1(A)$, the pull back of $\bP(\cF_i)$ by a translation $t_x : A \to A$ is in fact isomorphic to $\bP(\cF_i)$. An isomorphism $\bP(\cF_i) \simeq t_x^*\bP(\cF_i) = \bP(t_x^*\cF_i)$ comes from a line bundle $\alpha$ (depending on $x$) and an isomorphism $t_x^* \cF_i \simeq \cF_i \otimes \alpha$ of vector bundles. Thus $\cF_i$ is semihomogeneous.

	Therefore the vector bundle $\cE$ is an iterated extension of stable semihomogeneous vector bundles, all of the same slope class.  By \cite[Prop.~6.17]{Mukai1978:SemihomogeneousVectorBundles} nonisomorphic simple semihomogeneous vector bundles $\cF \not\simeq \cF'$ of the same slope class have no nontrivial extensions:
	\[
	\Ext^1(\cF,\cF') = \rH^1(A, \cHom(\cF,\cF')) = 0.
	\]
	It follows by induction on the length $k$ of the filtration that $\cE$ is a direct sum $\cE = \bigoplus_{j=1}^s \cE_j$ of vector bundles $\cE_j$ that are iterated extensions of simple semihomogeneous vector bundles $\cG_j$, with the $\cG_j$ pairwise nonisomorphic but all of the same slope class $\mu(\cE)$. In the language of \cite[Prop.~6.4]{Mukai1978:SemihomogeneousVectorBundles} the vector bundle $\cE_j$ is $\cG_j$-potent and thus semihomogeneous of slope class $\mu(\cE)$. By \cite[Prop.~6.9]{Mukai1978:SemihomogeneousVectorBundles} the category of semihomogeneous vector bundles with the same slope class $\mu(\cE)$ is closed under direct summands. Thus $\cE$ is semihomogeneous as claimed. 
\end{proof}

\subsection{Nakano positivity}

Recall that a vector bundle $\cE$ is Nakano positive if there is a hermitian metric $h$ on $\cE$ for which the associated curvature form on $\cE$ is pointwise positive definite.

\begin{prop}
	\label{prop:proj-flat-nak}
	Let $X$ be a smooth projective complex variety, and let $(\cE,h)$ be a projectively flat hermitian vector bundle on $X.$  Then the following are equivalent:	
	\begin{enumerate}[label=(\alph*)]
		\item 
		\label{propitem:Nakano1} 
		$\cE$ is Nakano positive.
		\item 
		\label{propitem:Nakano2} 
		$\cE$ is ample.
		\item 
		\label{propitem:Nakano3} 
		$\det{\cE}$ is ample.
	\end{enumerate} 
\end{prop}

\begin{proof}
	If $\cE$ is ample, then any Schur functor applied to $\cE$ is also ample, in particular $\det(\cE)$ is ample. This shows \ref{propitem:Nakano2}  $\Rightarrow$ \ref{propitem:Nakano3}, and the implications \ref{propitem:Nakano1}  $\Rightarrow$ \ref{propitem:Nakano2} follows from 
	\cite[Lemma~1.5, Proposition~1.16]{Umemura1973:ResultsTheoryVector}. It remains to verify \ref{propitem:Nakano3} $\Rightarrow$ \ref{propitem:Nakano1}, and this is \cite[Lemma~2.3]{Umemura1973:ResultsTheoryVector}.
\end{proof}

\begin{prop}
	\label{prop:ample css implies Nakano}
	Let $\cE$ be an  curve semistable vector bundle on a smooth projective variety $X$ with $\det(\cE)$ ample.  Then $\cE$ is Nakano positive.
\end{prop}
\begin{proof}
	Indeed, the characterization in \tref{thm:projectively flat}  \ref{rmkitem:projectivelyflatfiltration}
	implies that $\cE$ admits a filtration 
	\[
	0 = \cE_0 \subset \cE_1 \subset \ldots \cdots \cE_{k-1} \subset \cE_k = \cE
	\]
	whose successive quotients $\cF_i := \cE_i/\cE_{i-1}$ are hermitian projectively flat and 
	$\mu(\cF_i) = \mu(\cE)$.  
	By \cite[Lemma~2.2]{Umemura1973:ResultsTheoryVector}
	extensions of Nakano positive vector bundles are also Nakano positive. Hence, by induction on $\rank(\cE)$,  it suffices to show that all $\cF_i$ are Nakano positive. By \pref{prop:proj-flat-nak}, the subquotient $\cF_i$ is Nakano positive if and only if $\det(\cF_i)$ is ample. But $\mu(\cF_i)$ equals $\mu(\cE)$, so these $\det(\cF_i)$  numerically are positive multiples of $\det(\cE)$, hence ample. 
\end{proof}

\subsection{Variations on a result of Miyaoka} 
\label{sec:Miyaoka}

Let $\cE$ be a vector bundle of rank $r$ on a smooth projective variety $S$. We will study the convex Fujita number of $X = \bP(\cE)$, the associated projective space bundle $\pi: \bP(\cE) \to S$. The Picard group of $\bP(\cE)$ sits in a short exact sequence
\[
0 \to \Pic(S) \xrightarrow{\pi^\ast} \Pic(\bP(\cE)) \xrightarrow{\res} \Pic(\bP^{r-1}) \to 0
\]
where $\res$ restricts to a fibre and $\Pic(\bP^{r-1} )= \bZ$. The sequence splits using the tautological line bundle $\dO(1)$ associated to $\cE$. Every line bundle on $\bP(\cE)$ is uniquely of the form $\pi^\ast \cM(a)$.

A classical result of Miyaoka for semistable vector bundles on curves \cite[Theorem 3.1]{Miyaoka1987:ChernClassesKodaira} can be extended to curve semistable vector bundles as follows. For  a generalization of  \cite[Theorem 3.1]{Miyaoka1987:ChernClassesKodaira} in a different direction by Misra and Ray we refer to \cite{MisraRay2022:NefConesProjective}.

\begin{prop}
\label{prop:NEFconePE}
Let $\pi:  \bP(\cE) \to S$ be the projective space bundle of a curve semistable vector bundle $\cE$ of rank $r$ on a smooth projective  variety $S$. A line bundle $\cL = \pi^\ast \cM (a)$ is nef (resp.\ ample) if and only if $a \geq 0$ (resp.\ $a > 0$) and 
\[
\cM^{\otimes r} \otimes (\det(\cE))^{\otimes a}
\]
is nef (resp.\ ample) on $S$. 
\end{prop}
\begin{proof}
Let $h: C \to \bP(\cE)$ be an arbitrary smooth projective curve mapping to $\bP(\cE)$. The intersection number $(\cL \bullet C) = \deg_C(h^\ast \cL)$ can be computed after base change along $f = \pi \circ h$ as an intersection number of a curve on the pull back projective space bundle  $\pi: \bP(f^\ast \cE) \to C$ with respect to the pull back of $\cL$ along the projection  $\pr: \bP(f^\ast \cE) \to \bP(\cE)$. This shows that $\cL$ is nef on $\bP(\cE)$  if and only if for all $f: C \to S$ the pull backs $\pr^\ast \cL = \pi^\ast(f^\ast \cM)(a)$ are nef line bundles on $\bP(f^\ast \cE)$. 

Since by assumption $f^\ast \cE$ is still semistable, Miyaoka  \cite[Theorem 3.1]{Miyaoka1987:ChernClassesKodaira}  determines the nef cone in $\NS(\bP(f^\ast \cE))_\bR$ as the cone  generated by pull backs of nef line bundles on $C$ and the normalized hyperplane class $\dO(r) \otimes \pi^\ast \det(f^\ast\cE)^{-1}$. This translates into the following formula for the nef cone:
\begin{align*}
\Nef(\bP(f^\ast \cE)) & = \langle \pi^\ast \Nef(C), \dO(r) \otimes \pi^\ast \det(f^\ast\cE)^{-1} \rangle_{\bR_{\geq 0}} \\
& =  \langle\{ \pi^\ast (\cM)(a) \ ; \  a \geq 0 \text{ and } \cM^{\otimes r} \otimes \det(f^\ast \cE)^{\otimes a} \in \Nef(C)\} \rangle_{\bR_{\geq 0}} .
\end{align*}
Therefore $\cL = \pi^\ast \cM(a)$ on $\bP(\cE)$ is nef if and only if $a \geq 0$ and for all $f: C \to S$ 
\[
f^\ast \big(\cM^{\otimes r} \otimes (\det(\cE))^{\otimes a}\big) = (f^\ast \cM)^{\otimes r} \otimes \det(f^\ast(\cE))^{\otimes a}
\]
is nef on $C$. This is equivalent to the claimed characterization of nefness for $\cL$. 

We will better understand the ample cone as the interior of the nef cone  after the following transformation. 
The map 
\[
\NS(\bP(\cE))  \xrightarrow{} \NS(S) \times \bZ, \quad \pi^\ast \cM (a) \mapsto \big(\cM^{\otimes r} \otimes \det(\cE)^{\otimes a}, a\big)
\]
induces upon scalar extension an orientation preserving linear isomorphism
\[
\NS(\bP(\cE))_{\bR} \xrightarrow{\sim} \NS(S)_{\bR} \times \bR
\]
which maps the nef cone of $\bP(\cE)$ bijectively onto the product 
\begin{equation}
\label{eq:nef Pcsst}
\Nef(\bP(\cE)) \xrightarrow{\sim} \Nef(S) \times \bR_{\geq 0}.
\end{equation}
The description of the ample cone of $\bP(\cE)$ is now clear.
\end{proof}

\section{Convex Fujita numbers of projective bundles} 
\label{sec:FujitaPE}

\subsection{Preliminaries}
We start with a lower bound for the convex Fujita number taking into account restrictions to fibres.

\begin{prop}
\label{prop:FNonPElowerbound}
Let $\cE$ be a vector bundle of rank $r \geq 2$ on a smooth projective variety $S$. Then $$
\conFN(\bP(\cE)) \geq r$$
\end{prop}
\begin{proof}
It suffices to exhibit an ample line bundle $\cL$ on $\bP(\cE)$ such that $\omega_{\bP(\cE)} \otimes \cL^{\otimes (r-1)}$ is not globally generated.  Let $\pi : \bP(\cE) \to S$ be the projection map.  Recall that 
\begin{equation}
\label{eq:canonical bundle formula}
	\omega_{\bP(\cE)} \cong \cO_{\bP(\cE)}(-r) \otimes \pi^{\ast}(\omega_{S} \otimes \det{\cE})
\end{equation} 
Let $\cM$ be an ample line bundle on $S$ such that $\cE \otimes \cM$ is an ample vector bundle.  Then $\cL := \cO_{\bP(\cE)}(1) \otimes \pi^{\ast}\cM \cong \cO_{\bP(\cE \otimes \cM)}(1)$ is an ample line bundle on $\bP(\cE),$ and
\[
	\omega_{\bP(\cE)} \otimes \cL^{\otimes (r-1)} \cong \cO_{\bP(\cE)}(-1) \otimes \pi^{\ast}(\omega_{S} \otimes \det{\cE} \otimes \cM^{\otimes (r-1)})
\]
Since the restriction of this line bundle to any fibre of $\pi$ (identified with $\bP^{r-1}$) is $\cO_{\bP^{r-1}}(-1),$ it cannot be globally generated.  Thus $\conFN(\bP(\cE)) \geq r$ as claimed.
\end{proof}

Next we show a criterion for global generation of adjoint bundles on projective bundles by making use of relative global generation for the projection.

\begin{lem}
\label{lem:FNonPEuppperboundCriterion}
Let $\cE$ be a vector bundle of rank $r \geq 2$ on a smooth projective variety $S$, and let $\pi: \bP(\cE) \to S$  be the associated projective bundle. Let $t \geq r$ and $\cL_i = \pi^\ast\cM_i(a_i)$ for $i=1, \ldots, t$ ample line bundles on $X=\bP(\cE)$. We abbreviate $\cM \coloneq \bigotimes_{i=1}^t \cM_i$ and $\cL \coloneq \bigotimes_{i=1}^t \cL_i$ and  $a \coloneq \sum_{i=1}^t a_i$, and moreover
\[
\cF \coloneq \Sym^{a-r}(\cE) \otimes \det(\cE) \otimes \cM.
\]
Then the adjoint bundle $\omega_X \otimes \cL$ is globally generated if $\omega_S \otimes \cF$ is globally generated on $S$.  
\end{lem}
\begin{proof} 
The restriction of $\cL_i$ to a fibre is isomorphic to $\dO(a_i)$. Since this is ample, we deduce $a_i \geq 1$ and $a \geq r$. In particular $\cF$ is well defined. The formula for the canonical bundle \eqref{eq:canonical bundle formula} and the projection formula show
\[
\pi_\ast(\omega_X \otimes \cL) = \omega_S \otimes \Sym^{a-r}(\cE) \otimes \det(\cE) \otimes \cM = \omega_S \otimes \cF. 
\]  
Fibrewise $\omega_X \otimes \cL$ restricts to $\dO(a-r)$ and thus is globally generated with trivial higher cohomology. Cohomology and base change shows that the ``relative global generation'' map 
\[
\pi^\ast \pi_\ast(\omega_X \otimes \cL) \surj \omega_X \otimes \cL,
\]
is surjecttive, because for $s \in S$ with fibre $X_s = \pi^{-1}(s)$ we have
\begin{align*}
\pi^\ast \pi_\ast(\omega_X \otimes \cL)|_{X_s} & = \dO_{X_s} \otimes \big(\pi_\ast(\omega_X \otimes \cL)|_s\big) =  \dO_{X_s} \otimes \rH^0\big(X_s, (\omega_X \otimes \cL)|_{X_s}\big) \\
& \simeq \dO_{X_s} \otimes \rH^0\big(X_s,\dO(a-r)\big) \surj  \dO(a-r) \simeq (\omega_X \otimes \cL)|_{X_s}.
\end{align*}
It follows that $\omega_X \otimes \cL$ is globally generated if $\pi_\ast(\omega_X \otimes \cL)$ is globally generated.
\end{proof}

With \lref{lem:FNonPEuppperboundCriterion} it becomes evident that we must understand the positivity properties of $\cF = \pi_\ast(\omega_{X/S} \otimes \cL)$. The following definition will be applied to $\mu(\cE)$. 

\begin{defi}
\label{defi:denominator}
Let $\delta \in \NS(X) \otimes \bQ$ be a rational class. The \textbf{denominator} of $\delta$ is defined to be the smallest $n > 0$ such that $n\delta$ lies in the image of $\NS(X) \to \NS(X) \otimes \bQ$. 
\end{defi}

The denominator of $\delta$ is the order of the cyclic group $\langle \delta,  \NS(X)/\tors \rangle/\NS(X)/\tors$. For $\mu(\cE)$ the denominator divides the rank $r$. If $\cE$ is curve semistable and the denominator equals $r$, then $\cE$ is stable. Indeed, if $\cE' \subseteq \cE$ is a nontrivial subbundle of the same rational slope class, then $\mu(\cE) = \mu(\cE')$ has the same denominator and thus divides $\rank(\cE') < r$.

\begin{prop}
\label{prop:positivity of F}
Let $\cE$ be a curve semistable vector bundle of rank $r \geq 2$ on a smooth projective variety $S$, and let $\pi: \bP(\cE) \to S$  be the associated projective bundle. Let $t \geq r$ and $\cL_i = \pi^\ast\cM_i(a_i)$ for $i=1, \ldots, t$ ample line bundles on $X=\bP(\cE)$. We abbreviate $\cM \coloneq \bigotimes_{i=1}^t \cM_i$ and $\cL \coloneq \bigotimes_{i=1}^t \cL_i$ and  $a \coloneq \sum_{i=1}^t a_i$, and moreover $\cF = \pi_\ast(\omega_{X/S} \otimes \cL)$. Then the following hold.
\begin{enumerate}
	\item
	\label{propitem:slope class}
	If $a \geq r$, the slope class of $\cF$ is
	$\mu(\cF) = \sum_{i=1}^t \big(a_i \mu(\cE) + \mu(\cM_i)\big) = a \mu(\cE) + \mu(\cM)$.
	\item
	\label{propitem:F larger than ample}	
	If $t>r$, then there is an ample divisor $\Theta$ on $S$ such that 
	$\cF(-\Theta)$ has an ample determinant. 
	\item 
	\label{propitem:F larger than ample at r}	
	If $t=r$, then there is an ample divisor $\Theta$ on $S$ such that 
	$\cF(-\Theta)$ has an ample determinant, unless we are in the critical case, i.e.,
	\begin{enumerate}[label=(\roman*)]
		\item $\mu(\cE)$ has denominator $r$, and 
		\item all $a_i$ are pairwise congruent modulo $r$ and coprime to $r$, and
		\item the ample line bundle $\det(\cE)^{\otimes a/r} \otimes \cM$ on $S$ is not the tensor product of two ample line bundles. 
	\end{enumerate}
	\item 
	\label{propitem:critical case}	
	If $t=r$ and we are in the critical case, then there is an ample divisor $\Theta$ on $S$ such that $\mu(\cF(-\Theta)) = 0$. 
\end{enumerate}
\end{prop}
\begin{proof}
\eqref{propitem:slope class} If $a \geq r$, then $\cF = \Sym^{a-r}(\cE) \otimes \det(\cE) \otimes \cM$ has slope
\[
\mu(\cF) = (a-r) \mu(\cE) + \mu(\det(\cE)) + \mu(\cM) = a \mu(\cE) + \mu(\cM) = \sum_{i=1}^t \big(a_i \mu(\cE) + \mu(\cM_i)\big).
\]

Note that, if $t \geq r$, then we have $a \geq r$ since $a_i > 0$ for all $i= 1, \ldots, t$.

\eqref{propitem:F larger than ample}	If $t > r$, we claim that there is a partition $I \cup J = \{1, \ldots, t\}$ with $I$ and $J$ nonempty such that $a_I = \sum_{i \in I} a_i$ is divisible by $r$.  We consider the $r$-many sums $b_\nu = \sum_{i=1}^\nu a_i$ for $1 \leq \nu \leq r$. If any of these is divisible by $r$, then the claim holds. Otherwise, the $b_\nu$ for $\nu=1, \ldots, r$ take only at most $r-1$ values modulo $r$. By the pigeon hole principle there are $1 \leq k < \ell \leq r$ with $b_k \equiv b_\ell$ modulo $r$. Then $I = \{k+1, \ldots, \ell\}$ has $a_I = b_\ell - b_k \equiv 0 \pmod r$. This proves the claim. 

Now we set $\Theta =  \det(\cE)^{\otimes a_I/r} \otimes \bigotimes_{i \in I} \cM_i$. This is a line bundle on $S$ with an ample slope class by \pref{prop:NEFconePE}. The slope of $\cF(-\Theta)$ computes as 
\[
\mu(\cF(-\Theta)) = \sum_{i \in J} a_j \mu(\cE) + \mu(\cM_j),
\]
which is also ample by \pref{prop:NEFconePE}. 

\eqref{propitem:F larger than ample at r}	We need to analyse the combinatorial argument. As long as the partition still exists, the conclusion of \eqref{propitem:F larger than ample} stands. If $\mu(\cE)$ has denominator less than $r$, then we can repeat the combinatorial argument with $r$ replaced by the denominator of $\mu(\cE)$ and  the conclusion of \eqref{propitem:F larger than ample} holds. Secondly, if not all $a_i$ are pairwise congruent modulo $r$, then we can rearrange without loss of generality and assume $a_1 \not\equiv a_2 \pmod r$. Then we consider the partial sums
\[
a_1, \ a_2, \ a_1 + a_2, \ a_1 + a_2 + a_3, \ \ldots, \ a_1 + \ldots + a_r.
\]
These provide $r+1$ many values modulo $r$, hence by the pigeon hole principle at least one residue occurs twice. This pair is not $a_1$ and $a_2$ by assumption. In all other cases we can again consider the difference and find the desired nontrivial partition.  

It follows that  the conclusion of \eqref{propitem:F larger than ample} stands unless all $a_i$ are pairwise congruent modulo $r$ and coprime to $r$. But then $r \mid a$ and $\Theta = \det(\cE)^{\otimes a/r} \otimes \cM$ is in fact an ample line bundle on $S$ with $\mu(\cF) = \mu(\Theta)$. If $\Theta$ is a sum  of two ample divisors $\Theta = \Theta_1 + \Theta_2$, then $\cF(-\Theta_1)$ has slope class that of $\Theta_2$ and the proof of \eqref{propitem:F larger than ample at r}	is complete. 

\eqref{propitem:critical case}	
We already saw that in the critical case $\mu(\cF) = \mu(\Theta)$ for an ample divisor $\Theta$. That shows $\mu(\cF(-\Theta) = 0$. 
\end{proof}

\subsection{Projective bundles over curves}
\label{sec:P bundle on curves}

The results of this section are contained in \cite[\S3]{ChenKuronyaMustopaStix2023:ConvexFujitaSurfaces} with the exception of \tref{thm:FNofPE Yusufs result} \eqref{propitem:FNonPEoverC2} \ref{propitem:FNonPEoverC2ii} 
 which is new.
 
Recall that a vector bundle $\cE$ on a smooth projective curve $C$ admits a unique Harder-Narasimhan filtration
\[
0 = \cE_0 \subseteq \cE_1 \subseteq \cE_2 \subseteq \ldots \subseteq \cE_{t-1} \subseteq \cE_t = \cE,
\]
by subbundles $\cE_i$ such that $\cE_i/\cE_{i-1}$ is semistable of slope $\mu_i$ with 
\[
\mu^+(\cE) \coloneq \mu_1 > \mu_2 > \ldots > \mu_{t-1} > \mu_t \eqcolon \mu^-(\cE).
\]
Since the slope of an extension is always in the interval of the slopes of its constituents, we find
\[
\mu^-(\cE) \leq \mu(\cE) \leq \mu^+(\cE)
\]
with equality if and only if $\cE$ is semistable. We call $\mu^+(\cE)$ (resp.\ $\mu^-(\cE)$) the \textbf{maximal} (resp.\ \textbf{minimal}) \textbf{slope} of $\cE$. Butler deduced the following proposition from  \cite[Theorem 3.1]{Miyaoka1987:ChernClassesKodaira}.

\begin{prop}[{\cite[Lemma 5.4]{Butler1994:NormalGenerationVector}}]
\label{prop:ample cone on PE over curve}
Let $\pi:  \bP(\cE) \to C$ be the projective space bundle of a vector bundle $\cE$ of rank $r \geq 2$ on a smooth projective  curve $C$. A line bundle $\cL = \pi^\ast \cM (a)$ is ample if and only if $a > 0$ and 
\[
\deg(\cM) + a \mu^-(\cE) > 0.
\]
\end{prop}

\begin{thm}
\label{thm:FNonPEuppperbound}
Let $\cE$ be a vector bundle of rank $r \geq 2$ on a smooth projective curve $C$. 
Then we have
\[
r \leq \conFN(\bP(\cE)) \leq r+1.
\]
If $\cE$ is not semistable then $\conFN(\bP(\cE)) = r$.
\end{thm}
\begin{proof}
The lower bound was established in \pref{prop:FNonPElowerbound}. The upper bound will be deduced from \lref{lem:FNonPEuppperboundCriterion}, we thus keep the notation of this lemma and its proof. 

Let $\cE^-$ be the quotient of $\cE$ which is the part of its Harder-Narasimhan filtration with minimal slope. Since symmetric powers and tensor products of semistable vector bundles are again semistable, see  \cite[Corollary 3.7 and 3.10]{Miyaoka1987:ChernClassesKodaira}, we find for $\cF = \Sym^{a-r}(\cE) \otimes \det(\cE) \otimes \cM$ that 
\[
\cF^-  = \Sym^{a-r}(\cE^-) \otimes \det(\cE) \otimes \cM.
\]
Let $r^-$ be the rank of $\cE^-$. By \pref{prop:ample cone on PE over curve}, we compute the minimal slope of $\cF$ as
\begin{align*}
\mu^-(\cF) & = \mu(\cF^-)  = \mu(\Sym^{a-r}(\cE^-)) + \mu(\det(\cE)) + \mu(\cM)  \\
& = (a-r) \mu^-(\cE) + r \mu(\cE) + \mu(\cM)  \\ 
& =  r  \big( \mu(\cE) - \mu^-(\cE)\big) + \sum_{i=1}^t \big(\deg(\cM_i) + a_i \mu^-(\cE) \big) \geq \frac{t}{r^-}.
\end{align*}
The latter holds, because $\mu(\cE) \geq \mu^-(\cE)$ and since $\deg(\cM_i) + a_i\mu^-(\cE) \geq \frac{1}{r^-}$ by \pref{prop:ample cone on PE over curve}. 

If $\cE$ is not semistable or $t \geq r+1$, we have $t > r^-$ and so  $\mu^-(\cF) > 1$. By \lref{lem:SlopeCriterionGloballyGeneratedVB} below, the adjoint bundle $\omega_C \otimes \cF$ is globally generated, and the claim on the convex Fujita number of $\bP(\cE)$  is a consequence of \lref{lem:FNonPEuppperboundCriterion}. 
\end{proof}

The following lemma is well known, see for example \cite[Lemma 1.12(3)]{Butler1994:NormalGenerationVector}.

\begin{lem}
\label{lem:SlopeCriterionGloballyGeneratedVB}
Let $C$ be a smooth projective curve, and let $\cF$ be a vector bundle on $C$ of minimal slope $\mu^-(\cF) > 1$. Then $\omega_C \otimes \cF$ is globally generated. 
\end{lem}
\begin{proof}
Let $P$ be a point in $C$. Global sections of $\omega_C \otimes \cF$ generate in $P$ if $\rH^1(C,\omega_C \otimes \cF(-P)) = 0$. By Serre duality, it suffices to show $\rH^0(C,\cF^\vee(P)) = 0$. The maximal slope of these dual coefficients is
\[
\mu^+(\cF^\vee(P)) = 1 - \mu^-(\cF) < 0.
\]
Any nontrivial map $\dO_C \to \cF^\vee(P)$ therefore violates the semistability of the filtration quotients of the Harder-Narasimhan filtration of $\cF^\vee(P)$. 
\end{proof}

Under suitable hypothesis we are able to decide the convex Fujita number 
among the two values that \tref{thm:FNonPEuppperbound} allows. 

\begin{thm}
\label{thm:FNofPE Yusufs result}
Let $\cE$ be a semistable vector bundle of rank $r \geq 2$ and degree $d$ on a smooth projective curve $C$. 
\begin{enumerate}[align=left,labelindent=0pt,leftmargin=*]
\item 
\label{propitem:FNonPEoverC1}
If $(r,d) \not= 1$, then $\conFN(\bP(\cE)) = r$.
\item 
\label{propitem:FNonPEoverC2}
If $(r,d) = 1$, then the following holds. 
\begin{enumerate}[label=(\roman*),align=left,labelindent=0pt,leftmargin=*]
\item 
\label{propitem:FNonPEoverC2i} 
If $d \not\equiv 1 \pmod r$ and for all $k > 0$ the bundle $\Sym^{rk}(\cE)$ has no direct summand that is a line bundle,  then $\conFN(\bP(\cE)) = r$.
\item 
\label{propitem:FNonPEoverC2ii} 
If $d \equiv 1 \pmod r$, then $\conFN(\bP(\cE)) = r+1$.
\end{enumerate}
\end{enumerate}
\end{thm}
\begin{proof}
We keep the notation of \tref{thm:FNonPEuppperbound}. The proof will recover arguments from \pref{prop:positivity of F} in a more concrete situation. 

We need to discuss the case $t = r$ for a semistable $\cE$ and decide whether there are $\cM_i$ and $a_i$  such that $\omega_C \otimes \cF = \pi_\ast(\omega_X \otimes \cL)$ is not globally generated. Note that $\cF$ is now also semistable and so $\mu(\cF) = \mu^-(\cF)$. The proof of \tref{thm:FNonPEuppperbound} shows that the only critical case is when $\mu(\cF) = 1$, and that happens exactly when 
\begin{equation}
\label{eq:criticalcasePEonCurve}
r \deg(\cM_i) + a_i d = 1, \text{ for all $i = 1, \ldots, r$},
\end{equation}
This is only possible if $r$ and $d$ are coprime. This shows assertion \eqref{propitem:FNonPEoverC1}.

\smallskip

We now assume $r$ and $d$ are coprime. Then, in particular, $\cE$ is stable. Moreover, we assume that $\cM_i$ and $a_i \geq 1$ are such that the equations \eqref{eq:criticalcasePEonCurve} hold. It follows that all the $a_i$ are congruent to each other modulo $r$. Thus $a-r$ is divisible by $r$.

For assertion \ref{propitem:FNonPEoverC2i} we now also assume that $d \not\equiv 1 \pmod r$. Then $a_i \geq 2$ and $a-r \geq r$. Recall from the proof of \tref{thm:FNonPEuppperbound} that it suffices to show that $\omega_C \otimes \cF$ for 
\[
\cF = \Sym^{a-r}(\cE) \otimes \det(\cE) \otimes \cM
\]  
is globally generated. The proof of \lref{lem:SlopeCriterionGloballyGeneratedVB}  shows that it suffices that for all points $P \in C$ the sheaf 
\[
\cF^\vee(P) \simeq \Sym^{a-r}(\cE^\vee) \otimes \det(\cE)^{-1} \otimes \cM^{-1}
\]
has no global sections.  By  \cite[3.2.11]{HuybrechtsLehn2010:GeometryModuliSpaces}, or more precisely \cite{RamananRamanathan1984:RemarksInstabilityFlag},
the bundle $\Sym^{a-r}(\cE^\vee)$ is polystable. Hence $\cF^\vee(P)$ is also polystable.
By assumption, none of the direct summands is a line bundle (note that $r$ divides $a-r$). By the slope computation from above, the bundle  $ 
\cF^\vee(P)$ is a polystable vector bundle of slope $0$, and so any global section would yield a direct summand that is a line bundle (of the form $\dO_C$), a contradiction. This shows assertion \ref{propitem:FNonPEoverC2i}.

We now prove assertion \ref{propitem:FNonPEoverC2ii}. By assumption $d = 1- kr$ for some $k \in \bZ$. Let $\cM_i$ be line bundles on $C$ of degree $k$, and let $a_i$ be equal to $1$ for all $i = 1, \ldots, r$. Then \eqref{eq:criticalcasePEonCurve} holds, and in particular $\cL_i = \pi^\ast \cM_i(1)$ is ample. More precisely we have
\[
\omega_X \otimes \cL = \dO(-r) \otimes \pi^\ast(\omega_C \otimes \det(\cE)) \otimes \pi^\ast \cM (r) = \pi^\ast(\omega_C \otimes \det(\cE) \otimes \cM).
\]
It follows that $\omega_X \otimes \cL$ is globally generated if and only if $\omega_C \otimes \cF = \omega_C \otimes \det(\cE) \otimes \cM$ is globally generated on $C$. Since $\mu(\cF) = 1$, by an appropriate choice of $\cM_i$ we may find $\cF = \dO_C (P)$ for a point $P \in C$. In this case $\omega_C \otimes \cF$ has $P$ as a base point, and this concludes the proof of assertion \ref{propitem:FNonPEoverC2ii}.
\end{proof}

\begin{rmk}
Semistable vector bundles of rank $r$ and degree $d$ on a curve $C$ of genus $\geq 2$ exist by 
\cite[Lemma 4.3]{NarasimhanRamanan1969:ModuliVectorBundles}. When $d$ and $n$ are coprime, then the semistable vector bundle $\cE$ is in fact stable. If $\cE$ is generic, then $\Sym^k(\cE)$ is again stable for all $k \geq 0$ by a result of Seshadri, 
see Hartshorne \cite[Theorem 10.5]{Hartshorne1970:AmpleSubvarietiesAlgebraic}. So vector bundles satisfying the respective assumptions in \tref{thm:FNofPE Yusufs result} exist in abundance. 

Balaji and Koll\'ar define the notion of a holonomy group for a stable vector bundle $\cE$ on an arbitrary smooth projective variety $S$. This is a reductive subgroup of the automorphism group $\GL(\cE(x))$ of a fibre $\cE(x) := \cE \otimes \kappa(x)$ which is minimal to contain all images of Narasimhan-Seshadri representations  associated to $\cE|_C$ for curves $x \in C \subseteq S$, as long as $\cE|_C$ is still stable.  In \cite[Corollary 6]{BalajiKollar2008:HolonomyGroupsStable} they prove that if the commutator subgroup of the holonomy group is either $\SL(\cE(x))$ or $\Sp(\cE(x))$, then $\Sym^k(\cE)$ is still stable for all $k \geq 0$. So with this assumption on holonomy of $\cE$  the assumptions of the case \ref{propitem:FNonPEoverC2i}  in \tref{thm:FNofPE Yusufs result}  holds provided $d \not\equiv 1 \pmod r$. 
\end{rmk}

\begin{rmk}
If $C$ has genus $0$, then the only projective space bundle of a semistable vector bundle agrees with $\bP^{r-1} \times \bP^1$.  Here by \cite{ChenKuronyaMustopaStix2023:ConvexFujitaFundgp}*{Cor.~2.8} we have (for $r \geq 2$)
\[
\conFN(\bP^{r-1} \times \bP^1) = \max\{\conFN(\bP^{r-1}),\conFN(\bP^1)\} = \max\{r,2\} = r.
\]
\end{rmk}

\subsection{Projective bundles of semihomogeneous bundles}
\label{sec:sh bundle on AV}

In this section the base variety will be an abelian variety $A$. Before stating our next result, we need a lemma.

\begin{lem}
\label{lem:semihom-gg}
	Let $\cF$ be a semihomogeneous vector bundle on an abelian variety $A$. If there is an ample divisor $\Theta$ on $A$ such that $\cF(-\Theta)$ has ample determinant, then $\cF$ is globally generated.
\end{lem}

\begin{proof}
Since line bundles are semihomogeneous, also the twist $\cF(-\Theta)$ is semihomogeneous. By 
\cite[Prop. 2.6]{KuronyaMustopa2020:ContinuousCMregularitySemihomogeneous} the bundle $\cF(-\Theta)$ and the ample line bundle $\dO_A(\Theta)$ are $\rM$-regular coherent sheafs (Mukai regular in the sense of Pareschi and Popa \cite{PareschiPopa2003:RegularityAbelianVarieties}). Then \cite[Theorem 2.4]{PareschiPopa2003:RegularityAbelianVarieties} shows that 
$\cF \cong \cF(-\Theta) \otimes \dO_A(\Theta)$ is globally generated as claimed.
\end{proof}

Recall \pref{prop:AV sh csst} saying that curve semistable vector bundles on $A$ are the same as semihomogeneous vector bundles. 

\begin{thm}
	\label{thm:proj-semihom}
	Let $A$ be an abelian variety of dimension $g \geq 1$, and let $\cE$ be a curve semistable 
	bundle on $A$ of rank $r \geq 2$.  Then the convex Fujita number of $X = \bP(\cE)$ satisfies
	\[
	r \leq \conFN(\bP(\cE)) \leq r+1.
	\]
	More precisely, the following holds.
	\begin{enumerate}[align=left,labelindent=0pt,leftmargin=*]
	\item 
	\label{propitem:FNonPEoverAVsemihom1}
	If the denominator of $\mu(\cE)$ is less than $r$, then $\conFN(\bP(\cE)) = r$.
	\item
	\label{propitem:FNonPEoverAVsemihom2} 
	If $\conFN(A) = 0$, i.e. all ample line bundles on $A$ are globally generated, then $\conFN(\bP(\cE)) = r$.
	\item 
	\label{propitem:FNonPEoverAVsemihom3} 
	If there is a line bundle $\cM$ on $A$ such that $\det(\cE) \otimes \cM^{\otimes r}$ is ample but not globally generated, then $\conFN(\bP(\cE)) = r+1$.
\end{enumerate}	
\end{thm}

\begin{proof}
The lower bound for $\conFN(\bP(\cE))$ was established in \pref{prop:FNonPElowerbound}. The upper bound will be deduced from \lref{lem:FNonPEuppperboundCriterion} and we thus keep the notation of this lemma and its proof. If $t > r$, then \pref{prop:positivity of F} provides an ample divisor $\Theta$ such that $\cF(-\Theta)$ has ample determinant. Now \lref{lem:semihom-gg} shows that $\cF$ is globally generated. As $\cF = \omega_A \otimes \cF$, the upper bound follows. 

\smallskip

\eqref{propitem:FNonPEoverAVsemihom1} If the denominator of $\mu(\cE)$ is less than $r$, then we are not in the critical case in the terminology of \pref{prop:positivity of F}. So  the same argument as above shows $\conFN(\bP(\cE)) = r$.

\smallskip

\eqref{propitem:FNonPEoverAVsemihom2} We only need to show global generation for $\omega_X \otimes \cL$ in the critical case, in particular $\cE$ is stable. By \pref{prop:positivity of F} we have $\mu(\cF(-\Theta)) = 0$ in the critical case, with $\Theta$ an ample divisor. By \cite[Thm.~4.17]{Miyaoka1987:ChernClassesKodaira} the bundle $\cF(-\Theta)$ is an iterated extension of line bundles of degree $0$. Thus $\cF$ is an iterated extension of line bundles $\cF_\alpha$ numerically equivalent to $\Theta$ and therefore ample.   
On the other hand,  \cite[3.2.11]{HuybrechtsLehn2010:GeometryModuliSpaces}, or more precisely \cite{RamananRamanathan1984:RemarksInstabilityFlag},
shows that $\Sym^{a-r}(\cE)$ is polystable. Hence $\cF$ is also polystable. Both properties together show that $\cF = \bigoplus_{\alpha} \cF_\alpha$ is a direct sum of ample line bundles. By assumption, all ample line bundles on $A$ are globally generated. Thus also $\cF$ is globally generated. This again shows 
$\conFN(\bP(\cE)) = r$.

\smallskip

\eqref{propitem:FNonPEoverAVsemihom3} We pick all $\cL_i = \pi^\ast \cM(1)$ for $i = 1, \ldots, r$. These line bundles are ample by  \pref{prop:NEFconePE} and by assumption. The sheaf $\cF$ in this case becomes the line bundle
\[
\cF =  \det(\cE) \otimes \cM^{\otimes r}.
\]
The surjective map $\pi^\ast \cF \to \omega_X \otimes \cL$ of line bundles is then an isomorphism. Thus $\omega_X \otimes \cL \simeq \pi^\ast \cF$ is globally generated if and only if $\cF$ is globally generated, which we assume it is not. 
\end{proof}

\begin{rmk}
Part \eqref{propitem:FNonPEoverAVsemihom3} of \tref{thm:proj-semihom} can be rewritten with the twist $\cE \otimes \cM$ as the new vector bundle $\cE$ on $A$ giving rise to the same projective bundle. Then  \eqref{propitem:FNonPEoverAVsemihom3} says: if $\cE$ is a curve semistable vector bundle of rank $r \geq 2$ on the abelian variety $A$ with $\det(\cE)$ ample but not globally generated, then $\cL = \dO_{\bP(\cE)}(1)$ on $X = \bP(\cE)$ provides an ample line bundle such that $\omega_X \otimes \cL^{\otimes r}$ is not globally generated and therefore $\conFN(X) = r+1$.
\end{rmk}

\begin{rmk}
\tref{thm:FNofPE Yusufs result} and \tref{thm:proj-semihom} both cover the case of an elliptic curve as a base. Let $\cE$ be a semistable vector bundle of degree $d$ and rank $r$ on an elliptic curve, then $(d,r)=1$ is equivalent to $\mu(\cE)$ having denominator $r$. So \tref{thm:FNofPE Yusufs result} \eqref{propitem:FNonPEoverC1} and 
 \tref{thm:proj-semihom} \eqref{propitem:FNonPEoverAVsemihom1} deal with the same cases. 
 Next, \tref{thm:FNofPE Yusufs result}\eqref{propitem:FNonPEoverC2}\ref{propitem:FNonPEoverC2i}  and 
\tref{thm:proj-semihom}  \eqref{propitem:FNonPEoverAVsemihom2} both never apply. 
Finally, \tref{thm:FNofPE Yusufs result}\eqref{propitem:FNonPEoverC2}\ref{propitem:FNonPEoverC2ii}  and 
\tref{thm:proj-semihom}  \eqref{propitem:FNonPEoverAVsemihom3} both again deal with the same cases and provide the same reason for $\conFN(\bP(\cE)) = r+1$.
\end{rmk}

\subsection{Projective bundles of curve semistable bundles and the Albanese map}
\label{sec:sh bundle on Albanese finite} 

In this section we will work with curve semistable vector bundles on a base $S$ such that line bundles are determined by line bundles on an abelian variety. Concretely, we ask that there is a map $\alpha \colon S \to A$ to an abelian variety $A$ such that $\alpha^\ast: \Pic(A) \to \Pic(S)$ is an isomorphism. It turns out that $\alpha$ must be finite and the Albanese map. Indeed, the map $\alpha$ factors over the Albanese map and the factorization $h: \Alb_S \to A$ is surjective with $h^\ast: \Pic^0(A) \to \Pic^0(\Alb_S) = \Pic^0(S)$ an isomorphism. Hence $h$ is an isogeny with trivial kernel, so $A = \Alb_S$. 

\begin{defi}
\label{defi:Albanese-finite}
A variety \textbf{\Albanesefinite}  is a smooth projective variety $S$ such that $\alpha^\ast \colon \Pic(\Alb_S) \to \Pic(S)$ is an isomorphism for the Albanese map $\alpha: S \to  \Alb_S$.
\end{defi}

\begin{lem}
Let $S$ be a variety \Albanesefinite. Then the Albanese map is a finite map.
\end{lem}
\begin{proof}
If there is curve in a fibre of the Albanese map, then its intersection number with an ample line bundle  $\cL$ must be positive.  On the one hand, because $\cL$ is the pull back of a line bundle on the Albanese variety, the intersection number also vanishes.  This shows that the Albanese map is quasi-finite, and hence finite. 
\end{proof}

\begin{ex}
\label{ex:abelianPicType1}
Cyclic covers of abelian varieties of dimension at least $4$ that are totally branched along a smooth ample divisor are examples of varieties \Albanesefinite, see  \cite{AngehrnSiu1995:EffectiveFreenessPoint}*{\S4.2}. Further examples are constructed by iterating this branched covering construction. 

In these examples of cyclic branched covers the Albanese map is also surjective and an isomorphism of fundamental groups, hence the fundamental group is abelian. 
\end{ex}

\begin{ex}
Another source of examples of varieties \Albanesefinite\ is provided by smooth hyperplane sections $S$ in abelian varieties $A$ of dimension at least $\dim(S) \geq 3$. The restriction to $S$ is an isomorphism on Picard groups by the Lefschetz hyperplane theorem for $\Pic(-)$, see  \cite{Grothendieck1965:CohomologieLocaleFaisceaux}*{Exp.~XII Cor.~3.6}. In these examples the Lefschetz hyperplane theorem also shows that the Albanese map is an isomorphism on fundamental groups. 
\end{ex}

\begin{prop}
\label{prop:gg for csst on Albanese-finite by PP theory}
Let $S$ be a variety \Albanesefinite. We further assume that  the  Albanese map is surjective or the Picard rank of $S$ equals $1$. Let $\cF$ be a curve semistable vector bundle on $S$ and $\Theta$ an ample divisor such that  $\mu(\cF(-\Theta))$  is an ample rational class. Then $\omega_S \otimes \cF$ is globally generated.
\end{prop}
\begin{proof}
Let $\alpha \colon S \to A$ be the Albanese morphism. Let $\lambda \in \Pic^0(A)$ be arbitrary. The twist $\cF(-\Theta) \otimes \alpha^\ast \lambda$ is also curve semistable. Since $\mu(\cF(-\Theta))$ is ample,  \pref{prop:ample css implies Nakano} shows that $\cF(-\Theta) \otimes \alpha^\ast \lambda$ is Nakano positive. Therefore $\omega_S \otimes \cF(-\Theta) \otimes \alpha^\ast \lambda$ has no higher cohomology. Hence $\alpha_\ast\big(\omega_S \otimes \cF(-\Theta) \big)$ is $\rM$-regular.

By assumption the line bundle $\Theta$ is of the form $\alpha^\ast(\tilde{\Theta})$. The divisor $\tilde{\Theta}$ is  necessarily ample on $A$ due to the Albanese map being finite surjective or the Picard rank being equal to $1$. 
Now ample divisors on abelian varieties are $\rM$-regular. Then \cite[Theorem 2.4]{PareschiPopa2003:RegularityAbelianVarieties} shows that 
\[
\alpha_\ast\big(\omega_S \otimes \cF(-\Theta) \big) \otimes \dO_A(\tilde{\Theta}) = \alpha_\ast\big(\omega_S \otimes \cF(-\Theta) \otimes \alpha^\ast\dO_A(\tilde{\Theta})\big) = \alpha_\ast(\omega_S \otimes \cF)
\]
is globally generated on $A$. Because $\alpha$ is a finite map, also $\omega_S \otimes \cF$ is globally generated.
\end{proof}

\begin{thm}
	\label{thm:FNonPE-Albanesefinite}
	Let $S$ be a variety \Albanesefinite. We further assume that  the  Albanese map is surjective or the Picard rank of $S$ equals $1$. Let $\cE$ be a curve semistable bundle on $S$ of rank $r \geq 2$ with associated projective bundle $\pi \colon X = \bP(\cE) \to S$. Then
	\[
	r \leq \conFN(X) \leq r+1
	\]
	and the following more precise statements hold:
\begin{enumerate}[align=left,labelindent=0pt,leftmargin=*]
	\item 
	\label{propitem:FNonPEoverAlbanesefinite1}
	If the denominator of $\mu(\cE)$ is less than $r$, then $\conFN(\bP(\cE)) = r$.
	\item 
	\label{propitem:FNonPEoverAlbanesefinite2} 
	If $\conFN(S) \leq 1$ and $\pi_1(S)$ is abelian, then $\conFN(\bP(\cE)) = r$.
	\item
	\label{propitem:FNonPEoverAlbanesefinite3} 	
	If there is a line bundle $\cM$ on $S$ such that $\det(\cE) \otimes \cM^{\otimes r}$ is ample but not globally generated, then $\conFN(\bP(\cE)) = r+1$.
\end{enumerate}
\end{thm}

\begin{proof}
The general estimate $r \leq \conFN(\bP(\cE)) \leq r+1$ and assertion \eqref{propitem:FNonPEoverAVsemihom1}  follow as in the proof of \tref{thm:proj-semihom} simply substituting 
\pref{prop:gg for csst on Albanese-finite by PP theory}  for  \lref{lem:semihom-gg}.

\smallskip

In the proof of assertion \eqref{propitem:FNonPEoverAVsemihom2} we also follow the proof of \tref{thm:proj-semihom} \eqref{propitem:FNonPEoverAVsemihom2}. We again find an ample $\Theta$ with $\mu(\cF(-\Theta)) = 0$. The simple curve semistable filtration quotients of $\cF(-\Theta)$ can be endowed with a hermitian flat metric, and not just a projectively flat one as in \tref{thm:projectively flat} \ref{rmkitem:projectivelyflatfiltration}. The corresponding unitary representation is irreducible and thus of dimension $1$ since $\pi_1(S)$ is assumed to be abelian. Thus $\cF$ is an iterated extension of line bundles $\cF_\alpha$ numerically equivalent to $\Theta$ and therefore ample.  As in the proof of \tref{thm:proj-semihom} we deduce that $\cF$ is in fact isomorphic to $\bigoplus \cF_\alpha$. Then $\omega_S \otimes \cF$ is globally generated because $\conFN(S) \leq 1$.  As usual,  \lref{lem:FNonPEuppperboundCriterion} completes the proof. 

\smallskip

Finally, assertion \eqref{propitem:FNonPEoverAVsemihom3} has the same proof as \tref{thm:proj-semihom} \eqref{propitem:FNonPEoverAVsemihom3}.
\end{proof}

\begin{rmks}
\begin{enumerate}
	\item
	Note that following the proof of \tref{thm:FNonPE-Albanesefinite} the case $r=1$ yields 
	$\conFN(S) \leq 2$ for $S$  \Albanesefinite\ and furthermore at least one of the two: 
	surjective Albanese map or Picard rank equal to $1$. 
	\item
	If $S$ is as in Example~\ref{ex:abelianPicType1} and the degree of the cyclic cover is 
	non-trivial (of degree at least $2$), then  
	\cite{ChenKuronyaMustopaStix2023:ConvexFujitaFundgp}*{Proposition~4.6} shows $\conFN(S) \leq 1$, because abelian varieties have convex Fujita number bounded above by $2$. In particular, then $\conFN(\bP(\cE)) =  r$ for curve semistable vector bundles $\cE$ of rank $r$ on such $S$. 
\end{enumerate}
\end{rmks}

\bibliographystyle{alpha}



\end{document}